\documentclass[12pt,english,spanish]{article}
\usepackage{amsmath}
\usepackage[T1]{fontenc}
\usepackage[letterpaper]{geometry}
\usepackage{amsmath}
\usepackage{amssymb}
\usepackage{comment}
\usepackage{graphicx}

\usepackage{latexsym}

\usepackage{epic}
\usepackage{pstricks}
\usepackage{pst-node}

\usepackage{pstricks-add}

\def\BBox{\kern  -0.2cm\hbox{\vrule width 0.2cm height 0.2cm}}

\renewcommand{\phi}{\varphi}

\newtheorem{lemma}{Lemma}[section]
\newtheorem{theorem}{Theorem}[section]

\title{Colorful Associahedra and Cyclohedra}
\author{
Gabriela Araujo-Pardo\thanks{Supported by CONACYT 166306, 178395; PAPIIT-M\'exico under project IN101912.  (garaujo@matem.unam.mx)}, 
Isabel Hubard\thanks{Supported by CONACYT 166951, PAPIIT-M\'exico under project IB101412 and by Para las Mujeres en la Ciencia L'Oreal-UNESCO-AMC 2012. (isahubard@im.unam.mx)}, \\
Deborah Oliveros\thanks{Supported by PAPIIT-M\'exico under project IN101912, CONACyT 166306. (dolivero@matem.unam.mx)},\\
{\small  Instituto de Matem\'{a}ticas}\\
{\small  Universidad Nacional Aut\'{o}noma de M\'{e}xico, M\'{e}xico}
\\[1.2ex]
Egon Schulte\thanks{Supported by NSF-grant DMS-0856675 and PAPIIT-M\'exico under project IN101912. (schulte@neu.edu)}\\
{\small Department of Mathematics}\\
{\small Northeastern University, Boston, USA} }

\begin{document}
\maketitle
\begin{center}
{\it Dedicated to the memory of Andrei Zelevinsky}
\end{center}

\begin{abstract}
Every $n$-edge colored $n$-regular graph $\mathcal{G}$ naturally gives rise to a simple abstract $n$-polytope, {\em the colorful polytope of $\mathcal{G}$\/}, whose $1$-skeleton is isomorphic to $\mathcal{G}$. The paper describes colorful polytope versions of the associahedron and cyclohedron. Like their classical counterparts, the colorful associahedron and cyclohedron encode triangulations and flips, but now with the added feature that the diagonals of the triangulations are colored and adjacency of triangulations requires color preserving flips. The colorful associahedron and cyclohedron are derived as colorful polytopes from the edge colored graph whose vertices represent these triangulations and whose colors on edges represent the colors of flipped diagonals.
\end{abstract}





\section{Introduction}

There are interesting connections between abstract polytopes and edge colored regular graphs. Every $n$-edge colored $n$-regular graph $\mathcal{G}$ naturally gives rise to a simple abstract $n$-polytope, called the colorful polytope of $\mathcal{G}$, whose $1$-skeleton is isomorphic to $\mathcal{G}$ and whose higher-rank structure is built from $\mathcal{G}$ by following precise instructions encoded in $\mathcal{G}$ (see \cite{colpoly}). In this paper we describe colorful polytope versions of two  well-known convex polytopes, the associahedron and the cyclohedron.

The associahedron is a simple convex polytope first described combinatorially by Stasheff~\cite{stash} in 1963. It is often called the Stasheff polytope. Its combinatorial structure was studied independently by Tamari~\cite{tamari} as a partially ordered set of bracketings of a non-associative product of factors. A realization of this structure as a convex polytope was discovered by Lee \cite{lee}, as well as by Haiman and Milnor in unpublished work, see also ~\cite{cfz} and ~\cite{post} for excellent references. There are other realizations of the associahedron, for instance, the realization of Shnider \& Sternberg~\cite{Shinder} using planar binary trees. See also Loday~\cite{loday} for an algorithm that uses trees to find realizations with integer coordinates. The classical associahedron arises as a special case of the more general construction of secondary polytopes due to Fomin \& Zelevinsky \cite{fom-zele}. These generalized associahedra also include the cyclohedron, a simple convex polytope first described as a combinatorial object in Bott \& Taubes~\cite{Bott} in connection with knot theory, and independently as a geometric polytope by Simion~\cite{simion}.

The ordinary associahedron and cyclohedron can be constructed as convex polytopes that encode triangulations and flips. Their  colorful polytope versions, the colorful associahedron and colorful cyclohedron, respectively, are abstract polytopes that similarly encode triangulations and flips, but now with the new feature that the diagonals of the triangulations are colored and adjacency of triangulations requires color preserving flips. The colorful associahedron and colorful cyclohedron are combinatorial coverings of the ordinary associahedron and cyclohedron, respectively.

The paper is organized as follows. We begin in Section~\ref{secbasnot} with a brief review of basic concepts for graphs and abstract polytopes. Then Sections~\ref{defcolas} and \ref{cov} investigate the combinatorial structure of the colorful associahedron and establish its covering relationship with the classical associahedon. In Section~\ref{autcolas} we determine the automorphism group of the colorful associahedron. Finally, Section~\ref{cyc} describes the structure of the colorful cyclohedron.

The last author wishes to dedicate this paper to the memory of his long-term colleague and friend, Andrei Zelevinsky, who recently passed away. 

\section{Basic notions}
\label{secbasnot}

We begin with a brief review of some terminology for graphs and abstract polytopes. See \cite{CL96,McMS02} for further basic definitions and results.

Let $\mathcal{G}$ be a finite simple graph (without loops or multiple edges), and let $V(\mathcal{G})$ denote its vertex set and  $E(\mathcal{G})$ its edge set. An {\em edge coloring\/} of $\mathcal{G}$ is an assignment of colors to the edges of $\mathcal{G}$
such that adjacent edges are colored differently; this is an {\em $n$-edge coloring} if $n$ colors are used. The minimum number $n$ for which $\mathcal{G}$ admits an $n$-edge coloring is called its {\em chromatic index\/} and is denoted by $\chi_{1}(\mathcal{G})$. 

Let $\Delta(\mathcal{G})$ denote the maximum degree among the vertices of $\mathcal{G}$. Clearly $\Delta(\mathcal{G})\leq\chi_{1}(\mathcal{G})$, and it was established in~\cite{Vizing} that $\chi _{1}(\mathcal{G})\leq \Delta (\mathcal{G})+1$. Thus finite graphs $\mathcal{G}$ either have $\chi_{1}(\mathcal{G})=\Delta (\mathcal{G})$ or $\chi_{1}(\mathcal{G})=\Delta(\mathcal{G})+1$. 
Throughout this paper we are mainly interested in connected $n$-regular graphs $\mathcal{G}$ of the first type, that is, graphs with $\chi_{1}(\mathcal{G})=\Delta (\mathcal{G})=n$. Any such ($n$-regular) graph, equipped with an $n$-edge coloring, is called a {\em properly edge colored graph\/}. 

An automorphism of a finite simple (uncolored) graph $\mathcal{G}$ is a permutation of the vertices of $\mathcal{G}$ that preserves the edges of $\mathcal{G}$. By $\Gamma(\mathcal{G})$ we denote the group of all automorphisms of $\mathcal{G}$. (This is a digression from \cite{CL96}.)  Now suppose $\mathcal{G}$ is a properly edge colored graph and $f\!: E(\mathcal{G})\mapsto R$ the underlying edge coloring map, where $R$ is the set of colors used. An automorphism $\gamma\in\Gamma(\mathcal{G})$ is called {\em color preserving} if $\gamma$ preserves the color of every edge of $\mathcal{G}$; that is, $f(\gamma(e))=f(e)$ for every edge $e$ of $\mathcal{G}$. More generally we say that $\gamma\in\Gamma(\mathcal{G})$ is {\em color respecting\/} if $f(\gamma(e))=f(\gamma(e'))$ whenever $e,e'$ are edges with $f(e)=f(e')$. Thus there are two interesting subgroups of $\Gamma(\mathcal{G})$ associated with the edge coloring of $\mathcal{G}$, namely the subgroup $\Gamma_p(\mathcal{G})$ consisting of all color preserving automorphisms of $\mathcal{G}$, and the subgroup $\Gamma_{c}(\mathcal{G})$ consisting of all color respecting automorphisms of $\mathcal{G}$. Clearly, $\Gamma_p(\mathcal{G})\leq \Gamma_c(\mathcal{G})\leq \Gamma(\mathcal{G})$.
\medskip

An ({\em abstract\/}) {\em polytope of rank\/} $n$, or simply an {\em $n$-polytope}, is a partially ordered set $\mathcal{P}$ with a strictly monotone rank function with range $\{-1,0,\ldots,n\}$ satisfying the following conditions. The elements of rank $j$ are called the {\em $j$-faces\/} of $\mathcal{P}$, or {\em vertices}, {\em edges\/} and {\em facets\/} of $\mathcal{P}$ if $j = 0$, $1$ or $n-1$, respectively. Each {\em flag\/} (maximal totally ordered subset) of $\mathcal{P}$ contains exactly $n + 2$ faces, including a unique minimal face $F_{-1}$ (of rank $-1$) and a unique maximal face $F_{n}$ (of rank $n$). Further, $\mathcal{P}$ is {\em strongly flag-connected\/}, meaning that any two flags $\Phi$ and $\Psi$ of $\mathcal{P}$ can be joined by a sequence of flags $\Phi = \Phi_{0},\Phi_{1},\ldots,\Phi_{l-1},\Phi_{l}= \Psi$, all containing $\Phi \cap \Psi$, such that $\Phi_{i-1},\Phi_{i}$ are {\em adjacent\/} (differ by exactly one face) for each $i=1,\ldots,l$. Finally, $\mathcal{P}$ satisfies the diamond condition, namely if $F$ is a $(j-1)$-face and $G$ a $(j+1)$-face with $F < G$, then there are exactly {\em two\/} $j$-faces $H$ such that $F<H<G$.

When $F$ and $G$ are two faces of an $n$-polytope $\mathcal{P}$ with $F \leq G$, we call $G/F := \{H \mid F \leq H \leq G\}$ a {\em section\/} of $\mathcal{P}$.  We usually identify a face $F$ with the section $F/F_{-1}$. For a face $F$, the section $F_{n}/F$ is said to be the {\em co-face\/} of $\mathcal{P}$ at $F$, or the {\em vertex-figure\/} at $F$ if $F$ is a vertex.  

An {\em automorphism} of a polytope $\mathcal{P}$ is a bijection of the faces of $\mathcal{P}$ that preserves the order. The {\em automorphism group\/} of $\mathcal{P}$ is denoted by $\Gamma(\mathcal{P})$.  

We call an abstract polytope of rank $n$ {\em simple\/} if its vertex-figures are isomorphic to $(n-1)$-simplices. This generalizes a well-known concept for convex polytopes (see \cite{Z95}).

Properly edge colored graphs naturally give rise to abstract polytopes. In fact, the following theorem was established in~\cite{colpoly}.

\begin{theorem}
\label{colpo}
Every finite connected properly edge colored $n$-regular graph $\mathcal{G}$ is the $1$-skeleton of a simple abstract polytope of rank $n$, the {\em colorful polytope\/} $\mathcal{P}_{\mathcal{G}}$ of $\mathcal{G}$.
\end{theorem}

The face structure of the colorful polytope $\mathcal{P}_{\mathcal{G}}$ of Theorem~\ref{colpo} is best described in terms of a family of equivalence relations on the vertex set of $\mathcal{G}$. As before let $f\!: E(\mathcal{G})\mapsto R$ denote the edge coloring map for $\mathcal{G}$. Given a subset of colors $C\subseteq R$ we define the equivalence relation $\sim_C$ on $V(\mathcal{G})$ as follows:\ if $v,w$ are vertices of $\mathcal{G}$, then $v\sim_{C} w$ if and only if there exists a path in $\mathcal{G}$ from $v$ to $w$ all of whose edges have colors in $C$. Then the faces of $\mathcal{P}_{\mathcal{G}}$ of non-negative ranks can be described as follows. 

The faces of rank $0$, the vertices, are just the vertices of $\mathcal{G}$. It is convenient to identify faces (of non-negative rank) with their vertex sets (subsets of $V(\mathcal{G})$), and simply define incidence between two faces by set-theoretic inclusion of their vertex sets. More precisely, for ranks $j=0,1,\ldots,n$, the vertex set of a typical $j$-face $F$ of $\mathcal{P}_{\mathcal{G}}$ is an equivalence class for the equivalence relation $\sim_C$ associated with a $j$-element subset $C$ of $R$. In other words, a typical $j$-face $F$ of $\mathcal{P}_{\mathcal{G}}$ can be represented as a pair $(C,v)$, where $v$ is a vertex, $C$ is $j$-subset of $R$, and $F$ is the equivalence class of $v$ under $\sim_C$. We simply write $F=(C,v)$. For $C=\emptyset$, that is $j=0$, we recover the vertices of $\mathcal{G}$; for $C=R$ we obtain $V(\mathcal{G})$ as (the vertex set of) the $n$-face of $\mathcal{P}_{\mathcal{G}}$. Finally, we simply append the empty set as the face of $\mathcal{P}_{\mathcal{G}}$ of rank $-1$.

It was proved in \cite{colpoly} that $\Gamma(\mathcal{P}_{\mathcal{G}})=\Gamma_{c}(\mathcal{G})$, that is, the automorphism group of the colorful polytope associated with $\mathcal{G}$ is just the color-respecting automorphism group of $\mathcal{G}$. 

\section{The colorful associahedron $\mathcal{A}^{c}_n$}
\label{defcolas}

The classical associahedron is an $n$-dimensional convex polytope constructed from the triangulations of a convex $(n+3)$-gon (for $n\geq 0$). In this section we describe the colorful associahedron, a colorful polytope variant covering the face-lattice of the ordinary associahedron. 

Suppose $N$ is a convex $(n+3)$-gon in the plane with vertices labeled $1,2,\ldots,n+3$ in cyclic order. A {\em triangulation\/} $t$ of $N$ is a tessellation of $N$ by triangles with vertices from among those of $N$. Any triangulation of $N$ consists of $n+1$ 
triangles and has $n$ diagonals (edges of triangles passing through the interior of $N$). We call a diagonal {\em short\/} or {\em long\/} if it joins vertices that are respectively two steps or more than two steps apart on the boundary of $N$. Let $\Delta_n$ denote the set of all triangulations of $N$. It is well-known (see~\cite{stan}) that the number of triangulations of a convex $(n+3)$-gon is given by the $(n+1)$-st Catalan number, $C_{n+1}$, so
\begin{equation}
\label{Catalan}
|\Delta_n| = C_{n+1} = \frac{1}{n+2} {{2n+2}\choose n+1}.
\end{equation}

If $t$ is a triangulation of $N$ and $d$ is a diagonal of $t$, then a new triangulation $t'=t'_{d}$ of $N$ can be obtained from $t$ 
by {\em flipping\/} $d$. More precisely, since $d$ is a common edge of two adjacent triangles in $t$ forming a convex quadrilateral, we can replace $d$ by the other diagonal $d'$ of this quadrilateral to obtain two new triangles meeting along $d'$, which, together with all other triangles of $t$, give the new triangulation $t'$ of~$N$.

The flipping operations for diagonals determine a graph $\mathcal{G}_n$ with vertex set $\Delta_n$, the {\em exchange graph\/}. In this graph, two vertices $t$ and $t'$ are {\em adjacent\/} if and only if, when viewed as triangulations, $t'$ is obtained from $t$ by flipping a diagonal of $t$. It is well-known that $\mathcal{G}_n$ is the $1$-skeleton of a simple convex $n$-polytope $\mathcal{A}_n$ known as the $n$-dimensional {\em associahedron\/}, or $n$-{\em associahedron\/} (see \cite{fom}). This polytope has $C_{n+1}$ vertices, one for each triangulation of $N$. The vertex-set of a typical $j$-face of $\mathcal{A}_n$ consists of all those vertices of $\mathcal{A}_n$ that, when viewed as triangulations, have a certain set of $n-j$ diagonals in common. In other words, the vertex-set of a typical $j$-face can be found from any of its vertices $t$ by fixing $n-j$ diagonals of $t$ and allowing all other diagonals to flip. 

The $2$-associahedron $\mathcal{A}_2$ is a pentagon. The $3$-associahedron $\mathcal{A}_3$ is a simple convex $3$-polytope with $6$ pentagonal faces and $3$ square faces (see Figure~\ref{aso6}). 

\begin{figure}[htbp]
\begin{center}
\includegraphics[width=55mm]{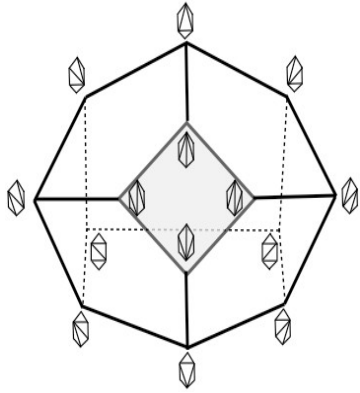} \\[-.2in]
\caption{\label{aso6}The 3-dimensional associahedron $\mathcal{A}_3$}
\end{center}
\end{figure}

The construction of the colorful associahedron proceeds in a similar way. However now, as an essential new feature, the diagonals of a triangulation (not the edges on the circumference of the underlying polygon) are colored and adjaceny of triangulations requires color preserving flips. 

As before, let $N$ be a convex $(n+3)$-gon with vertices $1,2,\ldots,n+3$. A ({\em diagonal-}) {\em colored triangulation\/} of $N$ is an ordinary triangulation of $N$ in which the diagonals are assigned a {\em color\/} from an $n$-set $C(N)$, the {\em color set\/}, such that any two distinct diagonals have distinct colors. Let $\Delta_{n}^{c}$ denote the set of all colored triangulations of $N$. Every ordinary triangulation of $N$ gives rise to $n!$ colored triangulations. Hence
\begin{equation}
\label{colCatalan}
|\Delta_{n}^c| = n!\cdot |\Delta_{n}| = n!\cdot C_{n+1} .
\end{equation}
This number $n!\cdot C_{n+1}$ is known as the $(n+1)$-st pseudo Catalan number (see~\cite{brualdi}). If $t$ is a colored triangulation and $d$ is any of its colored diagonals, we let $supp(t)$ and $supp(d)$, respectively, denote the underlying uncolored triangulation and uncolored diagonal of $N$, and call them the {\em support\/} of $t$ and~$d$. 

From now on, when the context is that of colored triangulations, we use the term ``diagonal'' without qualifications to mean ``colored diagonal''.  

If $t$ is a colored triangulation of $N$ and $d$ is a diagonal of $t$, then a new colored triangulation $t'$ (with a new diagonal $d'$) can be constructed from $t$ by {\em flipping\/} $d$ (while preserving the color of $d$). Here ``flipping'' is defined in the same way as before but now for colored diagonals. These flipping operations again determine a ({\em colorful\/}) {\em exchange graph\/}, now with vertex set $\Delta_{n}^{c}$ and denoted by $\mathcal{G}^{c}_n$. Two vertices $t$ and $t'$ of $\mathcal{G}^{c}_n$ are {\em adjacent\/} if and only if, when viewed as colored triangulations, $t'$ is obtained from $t$ by flipping a (colored) diagonal of $t$. As each triangulation has $n$ diagonals, it is clear that each vertex of $\mathcal{G}^{c}_n$ has 
degree~$n$. Moreover, if each edge of $\mathcal{G}^{c}_n$ is labeled with the color of the diagonal used in the flip to interchange the two vertices of the edge, then $\mathcal{G}^{c}_n$ becomes a properly edge colored graph with color set $R=C(N)$. Thus the following lemma holds. 

\begin{lemma}
\label{asociacolor}
$\mathcal{G}^{c}_n$ is an $n$-regular graph with chromatic index $\chi _1(\mathcal{G}^{c}_n)=n$.
\end{lemma}

The three smallest values of $n$ give rise to the following exchange graphs. When $n=0$ the base polygon $N$ is just a triangle, so $\mathcal{G}^{c}_{0}$ has just one vertex and no edges. When $n=1$ then $N$ is a quadrilateral and $\mathcal{G}^{c}_{1}$ has two vertices and one edge joining them. The situation is more interesting when $n=2$. Then $N$ is a pentagon and $\mathcal{G}^{c}_{2}$ a $10$-cycle (see Figure~\ref{exchange2}). Any pair of diametrically opposite vertices in $\mathcal{G}^{c}_{2}$ represents a pair of triangulations which differ only in their colors; that is, the underlying uncolored triangulations are identical but the colors on the two diagonals are switched. In other words, $\mathcal{G}^{c}_2$ is a $2$-fold covering of $\mathcal{G}_2$. The graphs $\mathcal{G}_n$ and $\mathcal{G}^{c}_n$ are isomorphic only when $n=0$ or $1$; this follows directly from (\ref{colCatalan}).

\begin{figure}[htbp]
\begin{center}
\scalebox{0.25}{\qquad\qquad\qquad\qquad\includegraphics{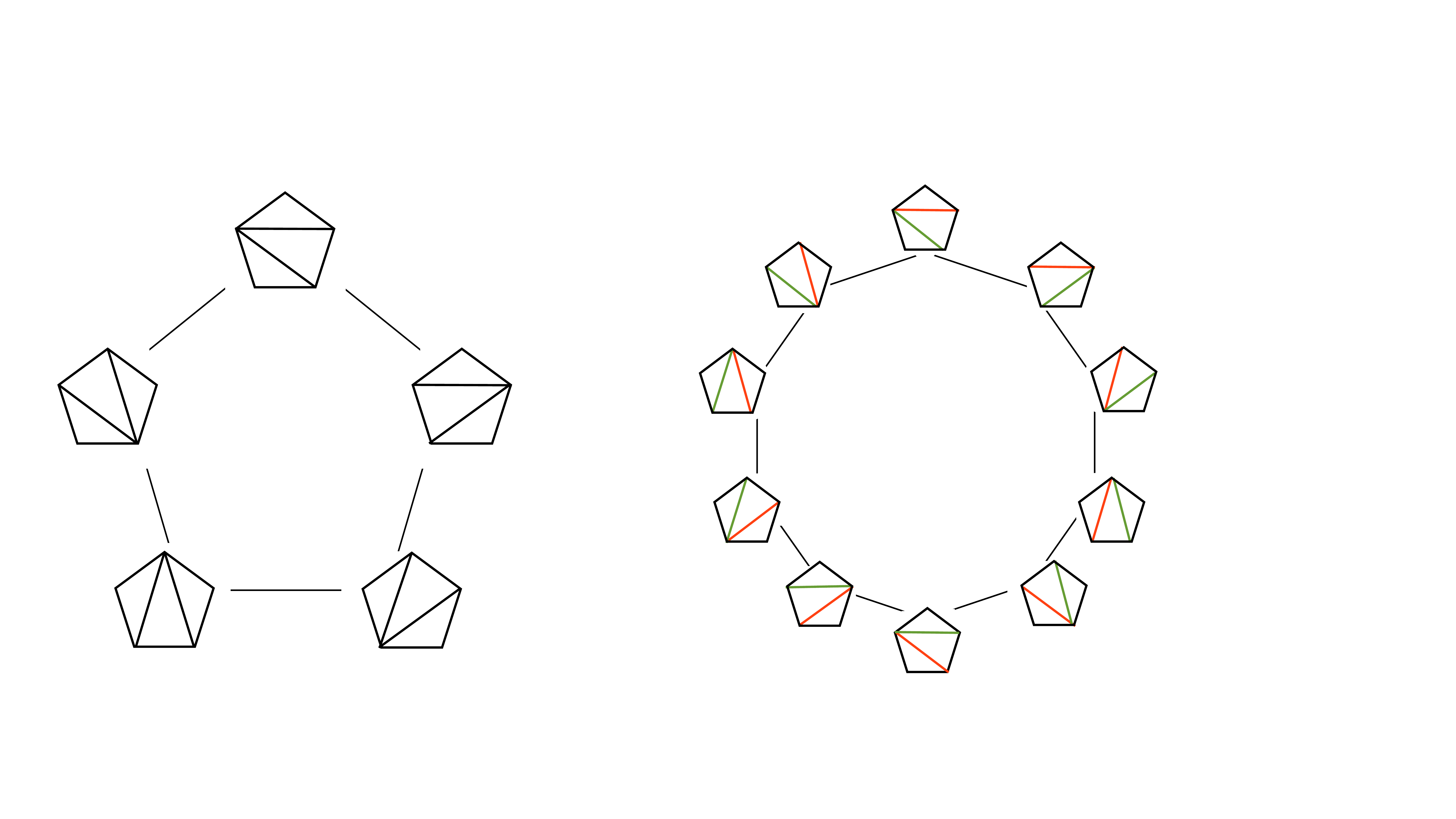}} \\[-.3in]
\caption{\label{exchange2} The exchange graphs $\mathcal{G}_{n}$ and $\mathcal{G}^{c}_{n}$ when $n=2$.}
\end{center}
\end{figure}

\begin{lemma}
\label{gconn}
$\mathcal{G}^{c}_n$ is connected for each $n$.
\end{lemma}

\begin{proof}
The proof is based on an induction argument. The above examples show that the cases $n=0,1,2$ are trivial, so let $n\geq 3$. Suppose that $\mathcal{G}^{c}_k$ is connected for each $k$ with $k<n$.

Now let $t$ and $s$ be two vertices of $\mathcal{G}^{c}_n$. Clearly, since the uncolored graph $\mathcal{G}_n$ is connected, we can find an edge path in $\mathcal{G}_n$ that joins $supp(t)$ and $supp(s)$. The edges in this path correspond to flips of diagonals in uncolored triangulations. When the very same flips are performed on the corresponding colored triangulation $t$ and its successors to give an edge path in $\mathcal{G}^{c}_n$, we see that the original vertex $t$ of $\mathcal{G}^{c}_n$ can be joined in $\mathcal{G}^{c}_n$ to a new vertex with support $supp(s)$. This new vertex will generally be different from $s$ but does have the same support as $s$. These arguments show that we may assume from now on that $t$ and $s$ have the same support.

Now suppose $t$ and $s$ have the same underlying triangulation but possibly different colorings of the diagonals. Identifying the $n!$ possible colorings on this triangulation with the elements of the symmetric group $S_n$, and using the fact that $S_n$ is generated by transpositions, we see that any coloring can be moved into any other coloring by successively switching exactly two colors at a time. If we can realize any such switch of two colors by a path between vertices in $\mathcal{G}^{c}_n$, then the concatenation of these paths in $\mathcal{G}^{c}_n$ gives a path joining $t$ and~$s$. Hence it suffices to assume from now on that $t$ and $s$ have the same support and that their colorings differ in only two positions, namely in the colors of the diagonals $d_1$ and $d_2$ (say) of $supp(t)=supp(s)$. 

If $d_1$ or $d_2$ is a long diagonal of $supp(t)$, we can truncate $N$ at some vertex $j$ by a short diagonal (colored differently than $d_1$ and $d_2$) to obtain an $(n+2)$-gon $\widehat{N}$, such that $t$ and $s$, respectively, induce colored triangulations $\widehat{t}$ and $\widehat{s}$ of $\widehat{N}$ which still have $d_1$ and $d_2$ as diagonals. Note here that $n\geq 3$. By induction applied to the graph $\mathcal{G}^{c}_{n-1}$ for $\widehat{N}$, we can connect the two corresponding vertices $\widehat{t}$ and $\widehat{s}$ by an edge path in $\mathcal{G}^{c}_{n-1}$. Now, if we append to each triangulation in this path the triangle that had been cut off from $N$ at $j$, we obtain an edge path in the original graph $\mathcal{G}^{c}_{n}$ joining $t$ and $s$. Thus the vertices $t$ and $s$ can be connected in~$\mathcal{G}^{c}_{n}$.

It remains to consider the case when both $d_1$ and $d_2$ are short diagonals of $supp(t)$. Now bear in mind that $n\geq 3$. In the special case when $n=3$ and $supp(t)$ is a triangulation of the hexagon with only short diagonals, we can directly appeal to Figure~\ref{exchange2} to see that $t$ and $s$ are connected in~$\mathcal{G}^{c}_{n}$; in fact, truncating the hexagon by the diagonal of $supp(t)$ distinct from $d_1$ and $d_2$ results in a pentagon with two induced colored triangulations connected in $\mathcal{G}^{c}_5$. In all other instances there is at least one long diagonal, $d_3$ (say), in $supp(t)$, and so we can first switch the colors on $d_1$ and $d_3$, then on $d_3$ and $d_2$, and finally again on $d_1$ and $d_3$. By the previous case, each switch can be realized by an edge path between the corresponding vertices of $\mathcal{G}^{c}_{n}$. Hence the concatenated paths give an edge path joining $t$ and $s$ in $\mathcal{G}^{c}_{n}$. Thus $\mathcal{G}^{c}_{n}$ is connected.
\end{proof}

The two previous lemmas show that $\mathcal{G}^{c}_{n}$ is a connected properly edge-colored $n$-regular graph. Then the following theorem is a direct consequence of Theorem~\ref{colpo}.

\begin{theorem}
\label{colorpolytope}
The exchange graph $\mathcal{G}_{n}^c$ given by a convex $(n+3)$-gon is the $1$-skeleton of a simple abstract polytope of rank $n$, called the {\em colorful $n$-associahedron\/} $\mathcal{A}^{c}_n$.
\end{theorem}

Recall that the vertices of $\mathcal{A}^{c}_n$ are just those of the graph $\mathcal{G}_{n}^c$ (that is, all colored triangulations of $N$), so $\mathcal{A}^{c}_n$ has exactly $n!\,C_{n+1}$ vertices. Moreover, since $\mathcal{A}^{c}_n$ is simple and has rank $n$, it has exactly $(n!)^{2}\,C_{n+1}$ flags. (There are $n!$ flags containing a given vertex.) For $j=0,\ldots,n$, a typical $j$-face $F$ of $\mathcal{A}^{c}_n$ can be represented as a pair $(C,t)$, where $t$ is a colored triangulation of $N$ and $C$ is a $j$-subset of the underlying color set $C(N)$. In other words, $F$ (or rather, its vertex set) is the equivalence class of $t$ under the equivalence relation $\sim_C$ defined at the end of Section~\ref{secbasnot}; that is, $F$ consists of all vertices of $\mathcal{G}_{n}^c$ that can be connected to $t$ by an edge path in $\mathcal{G}_{n}^c$ that uses only edges with colors in $C$. Thus the vertices, edges, and $2$-faces of $\mathcal{A}^{c}_n$ have representations of the form $(\emptyset,t)$, $(\{c\},t)$, and $(\{c,c'\},t)$, where $t$ is a vertex of $\mathcal{G}_{n}^c$ and $c,c'\in C(N)$, $c'\neq c$. In particular, a  vertex $(\emptyset,t)$ can be identified with $t$; an edge $(\{c\},t)$ is given by $\{t,t'\}$, where the triangulation $t'$ is obtained from $t$ by flipping the diagonal of color $c$; and a $2$-face $(\{c,c'\},t)$ consists of all those triangulations obtained from $t$ by flipping only diagonals with colors $c$ or~$c'$. 

The case $n=3$ describes the triangulations of the hexagon. The ordinary associahedron $\mathcal{A}_3$ is the convex polyhedron shown in Figure~\ref{aso6}. It has 14 vertices and 9 faces, namely 6 pentagons and 3 quadrilaterals. The colorful associahedron  $\mathcal{A}^{c}_3$ is a tessellation with $3$-valent vertices on an orientable surface of genus $4$, with $18$ decagonal and $18$ quadrilateral faces. It has $84$ vertices, of which $12$ are surrounded by three decagons and $72$ by two decagons and one quadrilateral. 
 
\section{Covering relationship between $\mathcal{A}^{c}_n$ and $\mathcal{A}_n$}
\label{cov}

In this section we show that the ordinary associahedron $\mathcal{A}_n$ is a quotient of the colorful associahedron $\mathcal{A}^{c}_n$ by a suitable subgroup of its automorphism group $\Gamma(\mathcal{A}^{c}_n)$. We already observed this relationship for $n=2$, where $\mathcal{A}^{c}_2$ is a decagon doubly covering $\mathcal{A}_2$, a pentagon (see Figure~\ref{exchange2}).

Consider the symmetric group $S_n$, viewed as acting on the underlying $n$-element color set $C(N)$. Then each permutation $\sigma\in S_n$ induces an automorphism of 
$\mathcal{A}^{c}_n$, again denoted by $\sigma$, defined on a face $(C,t)$ of $\mathcal{A}^{c}_n$ by
\begin{equation}
\label{actionsn}
\sigma((C,t)) := (\sigma(C),t_\sigma),
\end{equation}
where $t_\sigma$ denotes the colored triangulation (with the same support as $t$) obtained from $t$ by replacing each color $c$ by $\sigma(c)$ for $c\in C(N)$. It is straightforward to check that $\sigma$ respects the partial order on $\mathcal{A}^{c}_n$, so it is indeed an automorphism of $\mathcal{A}^{c}_n$. The subgroup of $\Gamma(\mathcal{A}^{c}_n)$ arising in this way is again denoted by $S_n$. 

Now consider the quotient $\mathcal{A}^{c}_n/S_n$ of the colorful associahedron $\mathcal{A}^{c}_n$ defined by the subgroup $S_n$ of its automorphism group $\Gamma(\mathcal{A}^{c}_n)$ (see \cite[2D]{McMS02} for the basic properties of quotients). Its set of faces consists of the orbits of faces of $\mathcal{A}^{c}_n$ under $S_n$; any two such orbits give incident faces of the quotient if and only if they can be represented as orbits of two incident faces of~$\mathcal{A}^{c}_n$. 

Now consider the orbit of a face $(C,t)$ of $\mathcal{A}^{c}_n$ under $S_n$; this consists of all faces $(\sigma(C),t_\sigma)$ of $\mathcal{A}^{c}_n$ with $\sigma\in S_n$. Clearly, $supp(t_\sigma)=supp(t)$ for each $\sigma\in S_n$, so all faces in the orbit of $(C,t)$ have $supp(t)$ as underlying uncolored triangulation. Moreover, since the vertex-set of $(C,t)$ is just the set of colored triangulations obtained from $t$ by performing only flips of diagonals with colors from $C$, we see that we can just as well think of the vertices of the face $(\sigma(C),t_\sigma)$ as the set of images of exactly these colored triangulations under the automorphism $\sigma$ of $\mathcal{A}^{c}_n$. In other words, $C$ and $\sigma(C)$ really determine the same sets of uncolored diagonals kept fixed in $supp(t)=supp(t_\sigma)$ when flipping operations are performed to construct $(C,t)$ or $(\sigma(C),t_\sigma)$, respectively; these uncolored diagonals are just those that carry the colors that are not in $C$ or not in $\sigma(C)$, respectively. Thus we may think of the quotient $\mathcal{A}^{c}_n/S_n$ as the ordinary associahedron~$\mathcal{A}_n$. 

In summary we could rephrase this by saying that {\em going colorblind on the colorful associahedron yields the ordinary associahedron\/}. More formally, $\mathcal{A}_n$ is a quotient of $\mathcal{A}^{c}_n$, or equivalently, $\mathcal{A}^{c}_n$ covers $\mathcal{A}_n$. In particular, we have proved the following theorem.

\begin{theorem}
\label{colorblind}
$\mathcal{A}^{c}_n/S_n \cong \mathcal{A}_n$.
\end{theorem}

\section{The automorphism group of $\mathcal{A}^{c}_n$}
\label{autcolas}

In this section we describe the automorphism group $\Gamma(\mathcal{A}^{c}_n)$ of the colorful associahedron $\mathcal{A}^{c}_n$ and relate it to the automorphism group $\Gamma(\mathcal{A}_n)$ of the ordinary associahedron $\mathcal{A}_n$. We first establish in Lemma~\ref{combaut} that $\Gamma(\mathcal{A}_n)$ is isomorphic to $D_{n+3}$, the automorphism group $\Gamma(N)$ of the underlying $(n+3)$-gon $N$ (see also Ceballos, Santos \& Ziegler~\cite{csz} for a different proof of this fact). Then, for the colorful associahedron, the subgroup $S_n$ of $\Gamma(\mathcal{A}^{c}_n)$ also comes into play. In fact, we already observed in (\ref{actionsn}) that $S_n$ can be viewed as a subgroup of $\Gamma(\mathcal{A}^{c}_n)$. We show in Theorem~\ref{colassocgroup}  that $\Gamma(\mathcal{A}^{c}_n)$ is the direct product of this subgroup $S_n$ and a subgroup isomorphic to $D_{n+3}=\Gamma(\mathcal{A}_n)$. 

The ordinary $n$-associahedron $\mathcal{A}_n$ contains exactly $n+3$ facets isomorphic to the $(n-1)$-associahedron $\mathcal{A}_{n-1}$, as can be seen as follows. For $j=1,\ldots,n+3$, let $N_j$ denote the $(n+2)$-gon obtained from $N$ by truncating vertex $j$. Each $N_j$ determines a facet $F_j$ of $\mathcal{A}_n$ isomorphic to $\mathcal{A}_{n-1}$ whose vertex-set consists of the triangulations of $N$ obtained from those of $N_j$ by adjoining the triangle $\{j-1,j,j+1\}$ (indices considered mod $n+3$). Moreover, every facet of $\mathcal{A}_n$ isomorphic to $\mathcal{A}_{n-1}$ is necessarily of this kind; in fact, only facets of $\mathcal{A}_n$ determined by short diagonals on $N$ have the correct number of vertices, $C_n$, for isomorphism with $\mathcal{A}_n$. Thus there are $n+3$ such facets. Note that each facet $F_j$ is adjacent to (that is, shares an $(n-2)$-face with) each facet $F_k$ with $k\neq j\pm1$, but that $F_j$ does not have a vertex in common with $F_{j-1}$ and $F_{j+1}$. In fact, when $k\neq j\pm 1$ the two diagonals $\{j-1,j+1\}$ and $\{k-1,k+1\}$ of $N$ determine a common $(n-2)$-face of $F_j$ and $F_k$ isomorphic to $\mathcal{A}_{n-2}$, but when $k=j\pm 1$ these diagonals cannot occur simultaneously as edges in a triangulation of $N$.

\begin{lemma}
\label{threecons}
Let $n\geq 2$ and $\gamma\in\Gamma(\mathcal{A}_n)$. If $\gamma$ fixes two facets $F_j$ corresponding to adjacent vertices of $N$, then $\gamma$ is the identity automorphism.
\end{lemma}

\begin{proof}
The group $\Gamma(\mathcal{A}_n)$ permutes the facets $F_1,\ldots,F_{n+3}$, since these are the only facets of $\mathcal{A}_n$ isomorphic to $\mathcal{A}_{n-1}$. Now suppose the two facets of $\mathcal{A}_n$ fixed by $\gamma$ are $F_1$, $F_2$ (say). Beginning with these two facets we then can work our way around the $(n+3)$-gon to show that $\gamma$ must fix every facet $F_j$ of the above kind. In fact, since $\gamma$ already fixes $F_2$ and $F_1$, it must also fix $F_3$, since $F_3$ is the only facet $F_j$ distinct from $F_1$ which does not meet $F_2$. Continuing in this fashion around the $(n+3)$-gon we then see that $\gamma$ must in fact fix every facet $F_{j}$. (We do not know at this point that $\gamma$ fixes all facets of $\mathcal{A}_n$.)

We now prove the lemma by induction on $n$. The statement is trivial for $n=2$, when $\mathcal{A}_2$ is a pentagon. (In this case $\gamma$ fixes all edges, so $\gamma$ is trivial.)

Now let $n \geq 3$. The facet $F_1$ is associated with the $(n+2)$-gon $N_1$ and its triangulations, and is isomorphic to $\mathcal{A}_{n-1}$.  Each of the vertices $2,\ldots,n+3$ of $N_{1}$ gives rise to a facet of $F_1$ isomorphic to $\mathcal{A}_{n-2}$. Let $(F_{1})_j$ denote the facet of $F_1$ for vertex $j$ associated with the $(n+1)$-gon $(N_{1})_{j}$ obtained by truncating $N_1$ at $j$.  Now, for $j=3,\ldots,n+2$, the two facets $F_1$ and $F_j$ of $\mathcal{A}_n$ are adjacent and meet precisely in this $(n-2)$-face $(F_{1})_j$. This observation allows us to complete the argument. In fact, since $\gamma$ fixes both $F_1$ and $F_j$, it must also fix their common $(n-2)$-face $(F_{1})_j$, for each such $j$. But $n\geq 3$, so $\gamma$ fixes such facets of $F_1$ corresponding to two adjacent vertices of $N_1$, for example, the vertices $3,4$. Hence, by the inductive hypothesis, the restriction of $\gamma$ to the facet $F_1$ is the identity automorphism on $F_1$. Now choose a flag $\Phi$ of $\mathcal{A}_n$ containing the facet $F_1$. As $\gamma$ is the identity on $F_1$, it must also fix the flag $\Phi$. Therefore, since automorphisms are uniquely determined by their effect on a single flag, $\gamma$ must necessarily be the identity automorphism of $\mathcal{A}_n$. 
\end{proof}

\begin{lemma}
\label{combaut} 
\ $\Gamma(\mathcal{A}_n) \cong D_{n+3}$
\end{lemma}

\begin{proof}
We know that $\Gamma(\mathcal{A}_n)$ permutes the facets $F_1,\ldots,F_{n+3}$. We prove that the permutation action of $\Gamma(\mathcal{A}_n)$ on $F_1,\ldots,F_{n+3}$ just corresponds to the standard action of the dihedral group $D_{n+3}$ on $1,\ldots,n+3$.

Clearly, the automorphism group $D_{n+3}$ of $N$ can be viewed as a subgroup of $\Gamma(\mathcal{A}_n)$ acting on the facets $F_1,\ldots,F_{n+3}$ in just the same way as $D_{n+3}$ does on $1,\ldots,n+3$.  By slight abuse of notation we also denote this subgroup of $\Gamma(\mathcal{A}_n)$ by $D_{n+3}$. Conversely, suppose $\gamma\in \Gamma(\mathcal{A}_n)$ and $\gamma(F_1)=F_j$ for some $j$. By applying (if need be) an automorphism of $\mathcal{A}_n$ from $D_{n+3}$ mapping $F_j$ back to $F_1$, we may assume that $j=1$, that is, $F_1$ is fixed by~$\gamma$. Since the facets $F_2$ and $F_{n+3}$ are the only facets $F_k$ which do not have a vertex in common with $F_1$, they are either fixed or interchanged by $\gamma$. Once again, by applying (if need be) an automorphism from $D_{n+3}$ we may further assume that $\gamma$ also fixes $F_2$ (and $F_{n+3}$). Now $\gamma$ fixes two facets associated with two adjacent vertices of $N$, namely $1$, $2$, and hence $\gamma$ is trivial by Lemma~\ref{threecons}. But now it is immediate that $\Gamma(\mathcal{A}_n)=D_{n+3}$; in fact, modulo $D_{n+3}$, each automorphism of $\mathcal{A}_n$ is equivalent to the identity.
\end{proof}

Proceeding with the colorful case we first observe that there are two distinguished subgroups of $\Gamma(\mathcal{A}^{c}_n)$. We already mentioned the subgroup $S_n$ acting as in~(\ref{actionsn}). The second subgroup is isomorphic to $\Gamma(N)=D_{n+3}$ and will again be denoted by $D_{n+3}$. In fact, every element $\gamma\in D_{n+3}$ naturally gives rise to an automorphism of $\mathcal{A}^{c}_n$ defined by
\begin{equation}
\label{groupan}
\gamma((C,t)) := (C,t_\gamma),
\end{equation}
and again denoted by $\gamma$, where $t_\gamma$ is the image of the colored triangulation $t$ under the automorphism $\gamma$ of $N$ (that is, $\gamma$ takes $supp(t)$ to $supp(t_\gamma)$ while preserving the colors on diagonals). Again it is straightforward to check that $\gamma$ is an automorphism of $\mathcal{A}^{c}_n$. 

First observe that $S_n$ centralizes $D_{n+3}$, that is, $\sigma\gamma=\gamma\sigma$ for all $\sigma\in S_n, \gamma\in D_{n+3}$. In fact, 
\[\sigma\gamma((C,t))=\sigma((C,t_\gamma))=(\sigma(C),(t_\gamma)_\sigma),\] 
and similarly,
\[ \gamma\sigma((C,t))=\gamma((\sigma(C),t_\sigma))=(\sigma(C),(t_\sigma)_\gamma),\] 
for all faces $(C,t)$ of $\mathcal{A}^{c}_n$. But $(t_\gamma)_\sigma=(t_\sigma)_\gamma$, since first applying an automorphism $\gamma$ of $N$ to $t$ and then permuting colors on diagonals of the resulting colored triangulation according to $\sigma$, is equivalent to first permuting colors on diagonals of $t$ according to $\sigma$ and then applying $\gamma$ to the resulting colored triangulation. 

Now since $S_{n}$ and $D_{n+3}$ intersect trivially, they generate a subgroup of $\Gamma(\mathcal{A}^{c}_n)$ isomorphic to $S_{n}\times D_{n+3}$. Our goal is to show that this subgroup is in fact $\Gamma(\mathcal{A}^{c}_n)$ itself.

We exploit the fact that $\mathcal{A}^{c}_n$ has exactly $n(n+3)$ facets isomorphic to $\mathcal{A}^{c}_{n-1}$, as can be seen as follows. Recall that $N_i$ denotes the $(n+2)$-gon obtained from $N$ by truncating vertex~$i$. Each such $N_i$ determines $n$ facets of $\mathcal{A}^{c}_n$ isomorphic to $\mathcal{A}^{c}_{n-1}$, one for each color $c\in C(N)$. We denote by $G_{i,c}$ the facet of $\mathcal{A}^{c}_n$ whose vertex-set consists of the triangulations of $N$ that have $\{i-1,i+1\}$ as a diagonal with color $c$; then $G_{i,c}$ corresponds to the set of triangulations of $N_i$ with diagonals colored from $C(N)\setminus\{c\}$. Furthermore, every facet of $\mathcal{A}^{c}_n$ isomorphic to $\mathcal{A}^{c}_{n-1}$ is determined by one such $(n+2)$-gon $N_i$ and one color $c$. Thus there are exactly $n(n+3)$ facets of $\mathcal{A}^{c}_n$ isomorphic to $\mathcal{A}^{c}_{n-1}$, namely the facets $G_{i,c}$.

Notice that an automorphism of $\mathcal{A}^{c}_n$ of the form $\sigma\gamma$, with $\sigma\in S_n$ and $\gamma\in D_{n+3}$, maps a facet $G_{i,c}$ of $\mathcal{A}^{c}_n$ to the facet $G_{\gamma(i),\sigma(c)}$. In particular, the elements of $D_{n+3}$ leave the second subscript $c$ of $G_{i,c}$ invariant, while those in $S_n$ keep the first subscript~$i$ fixed.

Define the families of facets 
\[ \mathcal{F}:=\{G_{i,c}\mid c\in C(N),i=1,\ldots,n+3\}\] 
and 
\[ \mathcal{F}_{i}:=\{G_{i,c}\mid c\in C(N)\}\quad (i=1,\ldots,n+3). \]

Note that, for any two distinct facets $G_{i,c}$ and $G_{j,c'}$, the number $K(i,c,j,c')$ of facets in $\mathcal{F}\setminus\{G_{i,c},G_{j,c'}\}$ that have both a vertex in common with $G_{i,c}$ and a vertex in common with~$G_{j,c'}$, is given by 
\begin{equation}
\label{Ks}
K(i,c,j,c') = \left\{ 
\begin{array}{ll}
(n-2)n & \mbox{if } j=i,\, c\neq c'\\ 
(n-1)(n-2) & \mbox{if } j=i\pm 1,\,c\neq c'\\
(n-1)^2 & \mbox{if }j=i\pm 1,\,c=c'\\
(n-2)^2 & \mbox{if }j=i\pm 2,\, c\neq c'\\
(n-1)(n-2) & \mbox{if }j=i\pm 2,\, c=c'\\
(n-2)(n-3) & \mbox{if }j\neq i,i\pm 1,i\pm 2,\, c\neq c'\\
(n-1)(n-3)& \mbox{if }j\neq i,i\pm 1,i\pm 2,\, c=c'
\end{array} 
\right.
\end{equation}
For example, for the first row, a facet $G_{k,c''}$ in $\mathcal{F}$ that has both a vertex in common with $G_{i,c}$ and a vertex in common with $G_{j,c'}=G_{i,c'}$, must have $k\neq i, i\pm 1$ and $c''\neq c,c'$. This gives $(n-2)((n+3)-3)=(n-2)n$ choices for $k$ and $c''$, and proves one direction of Lemma~\ref{useful}{b}; the other direction follows from a comparison of the numbers in~(\ref{Ks}). Note that the above relationships among facets in $\mathcal{F}$ (and hence the numbers in (\ref{Ks})) are preserved under automorphisms of $\mathcal{A}^{c}_n$. 

Notice in particular the following relationships among the facets in $\mathcal{F}$. The proofs are straightforward. 

\begin{lemma}
\label{useful}
Let $1\leq i,j\leq n+3$ and $c,c'\in C(N)$.\\[.01in]
(a) Two distinct facets $G_{i,c}$ and $G_{j,c'}$ share a common vertex if and only if $j\neq i,i\pm 1$ and~{$c\neq c'$}. In this case, $G_{i,c}$ and $G_{j,c'}$ are adjacent facets.\\[.01in]
(b) The number of facets in $\mathcal{F}\setminus\{G_{i,c},G_{j,c'}\}$ that have both a vertex in common with $G_{i,c}$ and a vertex in common with $G_{j,c'}$, is equal to $(n-2)n$ if and only if $i=j$ and $c\neq c'$.\\[.01in]
(c) No facet in $\mathcal{F}_i$ shares a vertex with a facet in $\mathcal{F}_j$ if and only if $|j-i|=1$.
\end{lemma} 

\begin{lemma}
\label{inducedaction}
The permutation action of $\Gamma(\mathcal{A}^{c}_n)$ on the family of facets $\mathcal{F}$ induces a permutation action on the set of families $\{\mathcal{F}_{1},\ldots,\mathcal{F}_{n+3}\}$. More precisely, each family $\mathcal{F}_{i}$ is a block of imprimitivity (in the sense of permutation group theory) for the action of $\Gamma(\mathcal{A}^{c}_n)$ on $\mathcal{F}$. 
\end{lemma}

\begin{proof} Let $1\leq i\leq n+3$, let $G_{i,c}$, $G_{i,c'}$ be two distinct facets in $\mathcal{F}_{i}$, and let $\gamma\in\Gamma(\mathcal{A}^{c}_n)$. Suppose $\gamma(G_{i,c})=G_{j,b}$ and $\gamma(G_{i,c'})=G_{k,b'}$. Since $c\neq c'$, we know that the facets $G_{i,c}$ and $G_{i,c'}$ themselves have no vertices in common; however, there are precisely $(n-2)n$ facets in $\mathcal{F}$ that have both a vertex in common with $G_{i,c}$ and a vertex in common with $G_{i,c'}$, namely the facets $G_{l,a}$ with $l\neq i,i \pm 1$ and $a\neq c, c'$. Since $\gamma$ is an automorphism of $\mathcal{A}^{c}_n$, this property must continue to hold for the images $G_{j,b}$ and $G_{k,b'}$ of $G_{i,c}$ and $G_{i,c'}$ under $\gamma$; that is, $K(i,c,i,c')=K(j,b,k,b')$. Hence, by (\ref{Ks}) this shows that $j=k$. Thus $\mathcal{F}_{i}$ is a block of imprimitivity for the action of $\Gamma(\mathcal{A}^{c}_n)$ on~$\mathcal{F}$.
\end{proof}

\begin{lemma}\label{threeconscolor}
Let $n\geq 2$ and $\gamma\in\Gamma(\mathcal{A}^{c}_n)$. If $\gamma$ fixes each facet in $\mathcal{F}_i$ for some $i$ (considered mod $n+3$), and $\gamma(\mathcal{F}_{i+1})=\mathcal{F}_{i+1}$, then $\gamma$ is trivial.
\end{lemma}

\begin{proof}
We know from Lemma~\ref{inducedaction} that $\gamma$ permutes $\mathcal{F}_1,\ldots,\mathcal{F}_{n+3}$. In particular, it follows from our assumptions that $\gamma$ fixes $\mathcal{F}_i$ and $\mathcal{F}_{i+1}$. First we show that $\gamma$ actually fixes each $\mathcal{F}_j$ for $j=1,\ldots,n+3$. To see this we begin with the two families $\mathcal{F}_i,\mathcal{F}_{i+1}$ invariant under $\gamma$, and work our way around $N$ to establish that $\gamma$ must leave each family $\mathcal{F}_j$ invariant. Here we exploit the fact that, given $j$ and $k$, no facet in $\mathcal{F}_j$ has a vertex in common with a facet in $\mathcal{F}_k$ if and only if $|k-j|=1$ (see Lemma~\ref{useful}). At the initial step (when $j=i+2$), we know that no facet in the family $\mathcal{F}_{i+2}$ has a vertex in common with a facet in $\mathcal{F}_{i+1}$. 
As $\gamma$ is an automorphism of $\mathcal{A}^{c}_n$ and $\mathcal{F}_{i+1}$ is invariant under $\gamma$, the corresponding property continues to hold for $\gamma(\mathcal{F}_{i+2})$ and $\gamma(\mathcal{F}_{i+1})=\mathcal{F}_{i+1}$. Thus we also must have $\gamma(\mathcal{F}_{i+2})=\mathcal{F}_{i+2}$. Continuing in this fashion around $N$ we then see that $\gamma$ fixes each $\mathcal{F}_{j}$.

We now use the first hypothesis on $\gamma$ to show that $\gamma$ fixes each facet in each family $\mathcal{F}_j$, that is, $\gamma(G_{j,c})=G_{j,c}$ for $j=1,\ldots,n+3$ and $c\in C(N)$. By assumption we already know this to hold for $j=i$, that is, $\gamma(G_{i,c})=G_{i,c}$ for each $c\in C(N)$. 
Now let $c\in C(N)$. Then the facets in $\mathcal{F}$ that have no vertex in common with $G_{i,c}$ comprise the facets $G_{j,c}$ with $j\neq i$, as well as the facets lying in $\mathcal{F}_{i-1}$, $\mathcal{F}_{i}\setminus\{G_{i,c}\}$ or $\mathcal{F}_{i+1}$. Now the corresponding statement remains true for the images under $\gamma$. Hence, since $\gamma$ fixes $G_{i,c}$ and leaves every family $\mathcal{F}_j$ with $j=1,\ldots,n+3$ invariant, this forces $\gamma(G_{j,c})=G_{j,c}$ for $j \neq i,i\pm 1$. As this also holds for $j=i$ (by assumption), $\gamma(G_{j,c})=G_{j,c}$ for each $j \neq i\pm 1$. Hence, since $c$ was arbitrary, $\gamma$ acts like the identity on each family $\mathcal{F}_j$ with $j \neq i\pm 1$. To extend this to the remaining families $\mathcal{F}_{i-1}$ and $\mathcal{F}_{i+1}$ we can simply replace the distinguished vertex $i$ of $N$ in the above by the vertex $i-3$ or $i+3$ (mod $n+3$), respectively. Thus $\gamma$ fixes each facet in each family $\mathcal{F}_j$.

We now complete the proof by induction on $n$. The case $n=2$ is clear. In this case $N$ is a pentagon and has only short diagonals; hence every facet of the decagon $\mathcal{A}_{2}^c$ lies in $\mathcal{F}$ and therefore is fixed by $\gamma$. Thus $\gamma$ is trivial. 

Now suppose $n\geq 3$. Choose a facet $G_{1,c}$ of $\mathcal{A}_{n}^c$ in $\mathcal{F}_{1}$ and keep it fixed. Recall that the vertices of $G_{1,c}$ are the triangulations of the $(n+3)$-gon $N$ that have $\{n+3,2\}$ as a diagonal with color $c$; these correspond to the triangulations of the $(n+2)$-gon $N_1$ with diagonals colored from $C(N_1):=C(N)\setminus\{c\}$. As a polytope, $G_{1,c}$ is isomorphic to the $(n-1)$-associahedron $\mathcal{A}_{n-1}^c$ derived from the colored triangulations of $N_1$ with colors from $C(N_1)$. Just as for $N$ and $\mathcal{A}_{n}^c$, the vertices $2,\ldots,n+3$ of $N_1$ naturally give rise to families of facets of this $(n-1)$-associahedron $\mathcal{A}_{n-1}^c$. For $i=2,\ldots,n+3$, let $G^{1}_{i,b}$ denote the facet of this $\mathcal{A}_{n-1}^c$ whose vertex-set consists of the triangulations of $N_1$ that have $\{i-1,i+1\}$ (now considered mod $n+2$) as a diagonal with color $b$; then $G^{1}_{i,b}$ corresponds to the set of triangulations of $(N_1)_i$ with diagonals colored from $C(N_1)\setminus\{b\}$. Also define 
\[ \mathcal{F}^{1}_{i}:=\{G^{1}_{i,b}\mid b\in C(N_1)\}\quad\, (i=2,\ldots,n+3). \]

Now consider $\gamma$. We know that $\gamma$ fixes each facet in $\mathcal{F}$ and, in particular, the facet $G_{1,c}$ of $\mathcal{A}_{n}^c$. Hence $\gamma$ acts (faithfully) as an automorphism on the $(n-1)$-associahedron $\mathcal{A}_{n-1}^c$ determined by $G_{1,c}$. Now consider a vertex $i$ of $N_1$ from among $3,\ldots,n+2$ (working again mod $n+2$). Every facet $G^{1}_{i,b}$ (with $b\neq c)$ of $\mathcal{A}_{n-1}^c$ in $\mathcal{F}^{1}_{i}$ is the common $(n-2)$-face of the pair of adjacent facets $G_{1,c}$ and $G_{i,b}$ of $\mathcal{A}_{n}^c$. Since $\gamma$ fixes both $G_{1,c}$ and $G_{i,b}$, it also fixes $G^{1}_{i,b}$. Hence $\gamma$ fixes every facet of $\mathcal{A}_{n-1}^c$ in $\mathcal{F}^{1}_{i}$ for 
$i=3,\ldots,n+2$. In particular, the automorphism $\gamma$ of $\mathcal{A}_{n-1}^c$ satisfies the assumptions of the lemma with $i=3$ and corresponding families $\mathcal{F}^{1}_{3},\mathcal{F}^{1}_{4}$ of facets of $\mathcal{A}_{n-1}^c$. Hence the inductive hypothesis applies and shows that $\gamma$ acts trivially on the facet $\mathcal{A}_{n-1}^c$ given by $G_{1,c}$. Thus $\gamma$ fixes an entire flag of the $n$-polytope $\mathcal{A}_{n}^c$, so $\gamma$ must be trivial.
\end{proof}

Now we can prove the following theorem.

\begin{theorem}\label{colassocgroup} 
$\Gamma(\mathcal{A}^{c}_n) \cong S_{n}\times D_{n+3}$
\end{theorem}

\begin{proof}
We need to show that the subgroup $S_{n}\times D_{n+3}$ of $\Gamma(\mathcal{A}^{c}_n)$ coincides with $\Gamma(\mathcal{A}^{c}_n)$.

Let $\gamma\in\Gamma(\mathcal{A}^{c}_n)$. Then $\gamma$ permutes $\mathcal{F}_1,\ldots,\mathcal{F}_{n+3}$, by Lemma~\ref{inducedaction}. Suppose that $\gamma(\mathcal{F}_1)=\mathcal{F}_i$ for some $i$. By applying an automorphism of $\mathcal{A}^{c}_n$ from $D_{n+3}$ that takes (the facets in) $\mathcal{F}_i$ back to (the facets in) $\mathcal{F}_1$, we may assume that $i=1$; that is, $\mathcal{F}_1$ is fixed by $\gamma$. 

We claim that then $\gamma$ either fixes or interchanges the two families $\mathcal{F}_2$ and $\mathcal{F}_{n+3}$. This is not obvious and can be seen as follows. Each of these two families has the property that none of its facets has a vertex in common with a facet in $\mathcal{F}_1$. Their images under $\gamma$ must have the same property, since $\gamma$ leaves $\mathcal{F}_1$ invariant. Now by Lemma~\ref{useful}, if $j$ and $k$ are such that no facet in $\mathcal{F}_j$ has a vertex in common with a facet in $\mathcal{F}_k$, then $|k-j|=1$. With $j=1$ and $k=2$ or $n+3$ this then settles our claim. 

By applying (if need be) an automorphism of $\mathcal{A}^{c}_n$ from $D_{n+3}$ that fixes $\mathcal{F}_1$ but interchanges $\mathcal{F}_2$ and $\mathcal{F}_{n+3}$, we may further assume that $\gamma$ leaves each of the families $\mathcal{F}_{n+3},\mathcal{F}_1,\mathcal{F}_{2}$ invariant. Note that at this point we have exhausted all possibilities for further applications of automorphisms from $D_{n+3}$.

Next we show that, modulo the subgroup $S_n$ of $\Gamma(\mathcal{A}^c_n)$, we can further achieve that $\gamma$ fixes every facet in $\mathcal{F}_1$. In fact, since $\gamma$ leaves $\mathcal{F}_1$ invariant, it permutes the facets $G_{1,c}$ in $\mathcal{F}_1$ (or rather, their colors) according to an element $\sigma$ in $S_n$. Now applying $\sigma^{-1}$ then reduces $\gamma$ to an automorphism that fixes every facet in $\mathcal{F}_1$. Hence we may assume that $\gamma$ itself fixes every facet of $\mathcal{A}^c_n$ in $\mathcal{F}_1$.

Thus $\gamma$ is an automorphism of $\mathcal{A}^c_n$ that satisfies the assumptions of Lemma~\ref{threeconscolor}. It follows that $\gamma$ is trivial.
\end{proof}

\section{The Colorful Cyclohedron}
\label{cyc}

In this final section we discuss a colorful polytope version of the cyclohedron. The cyclohedron is a convex $(n+1)$-polytope whose vertices correspond to the centrally symmetric triangulations of a given centrally symmetric convex $(2n+4)$-gon $K$ (see \cite{fom}).

Let $K$ be a centrally symmetric convex $(2n+4)$-gon with center $o$. A diagonal of $K$ is {\em central\/} if it passes through $o$. We only consider triangulations $t$ of $K$ which are centrally symmetric with respect to $o$. Every such triangulation $t$ has exactly one diagonal, the {\em central\/} diagonal of $t$, that passes through $o$. All other diagonals of $t$ occur in centrally symmetric pairs. There are $n$ such pairs and overall $2n+1$ diagonals in $t$.

In the colorful version we consider {\em colored triangulations\/} $t$ of $K$, meaning that all non-central diagonals receive a color from an $n$-set of colors $C(K)$ such that the two diagonals in a centrally symmetric pair are colored the same, and any two diagonals from distinct pairs are colored differently. Note that the central diagonal will {\em not\/} receive a color. (In principle we could also color the central diagonals of triangulations but then all central diagonals would have to receive the same color.)  Flips of diagonals of $t$ are defined similarly as before, except that now non-central diagonals must be flipped in pairs (while preserving the colors); that is, if a non-central diagonal is flipped, then so is the diagonal that forms a centrally symmetric pair with it. Thus, either the (uncolored) central diagonal is flipped (and becomes again a central diagonal), or the two diagonals in a centrally symmetric pair are flipped simultaneously. 

As for the colorful associahedron, the flipping operations determine the {\em (colorful) exchange graph\/}, denoted $\mathcal{H}^c_{n}$, whose vertices are the centrally symmetric triangulations of $K$. Now two vertices $t$ and $t'$ of $\mathcal{H}^{c}_n$ are {\em adjacent\/} if and only if, viewed as triangulations, $t'$ is obtained from $t$ by flipping either the central diagonal or both diagonals in a centrally symmetric pair. We show that $\mathcal{H}^{c}_n$ is the $1$-skeleton of a simple abstract $(n+1)$-polytope $\mathcal{Z}_{n}^c$ called the {\em colorful $(n+1)$-cyclohedron\/}. 

First note that the central diagonal in a centrally symmetric triangulation $t$ of $K$ naturally ``splits $t$ in half'', where each half is an ordinary triangulation of a convex $(n+3)$-gon $K^t$. (As the two halves are related by the central symmetry in $o$, we need to consider only one of them.) Thus every vertex $t$ of $\mathcal{H}^{c}_n$ determines (via its central diagonal) a vertex of the exchange graph $\mathcal{G}^{c}_n$ for the colorful $n$-associahedron. As there are $n+2$ choices for the central diagonal of $K$, the numbers of vertices of $\mathcal{H}^{c}_n$ and $\mathcal{G}^{c}_n$ are related by 
\begin{equation}
\label{vertcyc}
|V(\mathcal{H}^c_{n})|
=(n+2)\,|V(\mathcal{G}^{c}_{n})| 
=n!\!\cdot\! (n+2)C_{n+1}=n!\!\cdot\!\binom{2n+2}{n+1}.
\end{equation}
Note that the degree of a vertex $t$ in $\mathcal{H}^c_n$ is $n+1$, that is, one larger than the degree of the corresponding vertex in $\mathcal{G}^c_n$; the additional edge arises from the flip of the central diagonal in $t$. Thus $\mathcal{H}^{c}_n$ is an $(n+1)$-regular graph. In particular, the number of edges of $\mathcal{H}^{c}_n$ is given by 
\[ \frac{n+1}{2}\,|V(\mathcal{H}^c_{n})|  
=\frac{1}{2}(n+1)!\!\cdot\!\binom{2n+2}{n+1}
= \frac{(2n+2)!}{2(n+1)!} \,.\]

Next we establish the following lemma.

\begin{lemma}
\label{connectivitycic} 
$\mathcal{H}^{c}_{n}$ is connected for each $n$.
\end{lemma}

\begin{proof}
As in the proof of the connectedness of $\mathcal{G}^{c}_{n}$ (see Lemma~\ref{gconn}) we are taking for granted that the exchange graph for the ordinary cyclohedron is connected. Moreover, by arguments very similar to those for $\mathcal{G}^{c}_n$,
it suffices to show that any two vertices $s$ and $t$ of $\mathcal{H}^{c}_{n}$ with the same central diagonal can be connected by an edge path in $\mathcal{H}^{c}_{n}$. 

Suppose $s$ and $t$ are given. Let $s^*$ and $t^*$, respectively, denote the corresponding vertices of the colorful $n$-associahedron for the $(n+3)$-polygon $K^s$ determined by the common central diagonal of $s$ and $t$.  As the graph $\mathcal{G}^{c}_n$ derived from $K^s$ is connected, we can join the two vertices $s^*$ and $t^*$ by an edge path in $\mathcal{G}^{c}_n$. In other words, the triangulations $s^*$ and $t^*$ of $K^s$ can be moved into each other by a sequence 
of diagonal flips. When this sequence of flips on triangulations of $K^s$ is lifted to a sequence of corresponding flips on centrally symmetric triangulations of $K$, we obtain an edge path in $\mathcal{H}^{c}_{n}$ joining the original vertices $s$ and $t$. Thus $\mathcal{H}^{c}_{n}$ is connected. 
\end{proof}
    
At this point we know that $\mathcal{H}^{c}_{n}$ is a connected $(n+1)$-regular graph. In fact, $\mathcal{H}^{c}_{n}$ has the structure of a properly edge colored $(n+1)$-regular graph, in which $n$ of the $n+1$ edges emanating from a given vertex are labeled with the colors in $C(K)$ while the remaining edge (corresponding to the flip of the central diagonal) is uncolored, or ``colored'' with the {\em uncolor},~$c^*$.  Hence, appealing to Theorem~\ref{colpo} we obtain the following theorem.

\begin{theorem}
\label{cyclopolytope}
The exchange graph $\mathcal{H}_{n}^c$ given by a centrally symmetric convex $(2n+4)$-gon is the $1$-skeleton of a simple abstract polytope of rank~$n+1$, the {\em colorful $(n+1)$-cyclohedron\/}~$\mathcal{Z}^{c}_n$.
\end{theorem}

For example, when $n=1$ the polygon $K$ is a centrally symmetric hexagon and the colorful 2-cyclohedron $\mathcal{Z}^{c}_{1}$ is again a hexagon. When $n=2$ (and $K$ is a centrally symmetric octagon), the colorful $3$-cyclohedron $\mathcal{Z}^{c}_{2}$ (of rank $3$) is a trivalent map on the $2$-torus with 
$40$~vertices, $60$ edges and $20$ facets.

For $j=0,1,\ldots,n+1$, a typical $j$-face $F$ of the colorful $(n+1)$-cyclohedron $\mathcal{Z}^{c}_n$ can be represented as a pair $(C,t)$, where $t$ is a vertex of $\mathcal{Z}^{c}_{n}$ (a triangulation of $K$) and $C$ is a $j$-subset of the new $(n+1)$-element color set $C^{*}:=C(K)\cup\{c^*\}$, with the understanding that the vertex-set of $F$ in $\mathcal{Z}^{c}_n$ consists of all vertices representing centrally symmetric triangulations of $K$ obtained from $t$ by a sequence of flips involving only diagonals with colors from $C$. 

If $(C,t)$ is a facet (face of rank $n$) of $\mathcal{Z}^{c}_n$, the corresponding $n$-subset $C$ contains all colors in $C^*$, except one, called the {\it missing color\/} of the facet. A diagonal of $t$ carrying the missing color of a facet $(C,t)$ is called a {\em rigid\/} diagonal of this facet. Flips involved in generating the facet can only be performed on diagonals which are not rigid. A facet has one or two rigid diagonals, depending on whether or not the missing color is $c^*$. If there is just one rigid diagonal, then the diagonal of $K$ underlying this rigid diagonal is a central diagonal of $K$. On the other hand, if there are two rigid diagonals, then the two diagonals of $K$ underlying these rigid diagonals form a centrally symmetric pair of non-central diagonals of $K$. In short, a facet has either a central rigid diagonal or one centrally symmetric pair of non-central rigid diagonals. 

Among the facets with central rigid diagonals, a facet is uniquely determined by its rigid diagonal. However, among the facets with a centrally symmetric pair of non-central rigid diagonals, a facet is not generally determined by its pair of rigid diagonals (when $n>2$). In fact, the diagonal pair dissects $K$ into three regions, and flips involving non-rigid diagonals must preserve this dissection, preventing colors to move between regions. Thus a facet with a centrally symmetric pair of non-central rigid diagonals also depends on how the color set is partitioned among the regions (of course, the regions not containing the center of $K$ must receive the same colors). Note that, since $K$ has $2n+4$ edges and $n+2$ central diagonals (colorable only with $c^*$), $K$ has
\[ \binom{2n+4}{2}-(2n+4) -(n+2)=2n(n+2) \]
non-central diagonals and therefore $n(n+2)$ centrally symmetric pairs of non-central diagonals, each colorable with $n$ possible colors from $C(K)$.

The $n+2$ facets of $\mathcal{Z}^{c}_n$ with a central rigid diagonal are isomorphic to the colorful $n$-associahedron $\mathcal{A}^{c}_n$. In fact, a central diagonal divides $K$ into two halves, each an $(n+3)$-gon, and the centrally symmetric triangulations of $K$ with this central diagonal as diagonal are in one-to-one correspondence with the ordinary triangulations in one of these $(n+3)$-gons. Note that these $n+2$ facets partition the entire vertex-set of $\mathcal{Z}^{c}_n$. We let $F^d$ denote the facet of $\mathcal{Z}^{c}_n$ with central rigid diagonal $d$.

On the other hand, when $n>1$ the facets of $\mathcal{Z}^{c}_n$ with a centrally symmetric pair of non-central rigid diagonals are not isomorphic to $\mathcal{A}^{c}_n$, as can be seen as follows. (When $n=1$ the facets have rank $1$ and are trivially isomorphic to $\mathcal{A}^{c}_n$.) Suppose the diagonal pair dissects $K$ into three regions such that the central region $K'$ is a $(2k+4)$-gon. There are two cases to consider. First, if $k>0$ then the corresponding facet of $\mathcal{Z}^{c}_n$ must contain a $2$-face isomorphic to $\mathcal{Z}^{c}_1$, that is, a hexagon; in fact, take a centrally symmetric hexagon $K''$ with center $o$ and vertices among those of $K'$, triangulate the complements of $K''$ in $K'$ and of $K'$ in $K$ in a centrally symmetric fashion, and then allow only the diagonals of $K''$ to flip, thus creating a $2$-face of $\mathcal{Z}^{c}_n$ based on $K''$ and hence isomorphic to $\mathcal{Z}^{c}_1$. However, as we saw earlier, $\mathcal{A}^{c}_n$ has only squares and decagons as $2$-faces, so the facet of $\mathcal{Z}^{c}_n$ determined by the given diagonal pair cannot be isomorphic to $\mathcal{A}^{c}_n$. This then leaves the case $k=0$. In this case the central region $K'$ is a square, so a simple vertex count shows that the corresponding facet of $\mathcal{Z}^{c}_n$ cannot be isomorphic to $\mathcal{A}^{c}_n$; in fact, the number of vertices in the facet is just twice that of the colorful $(n-1)$-associahedron $\mathcal{A}^{c}_{n-1}$, and hence equals the number of vertices of the colorful $n$-associahedron $\mathcal{A}^{c}_{n}$ only when $n=1$. This settles our claim.
\medskip

As in the case of the colorful associahedron, the symmetric group $S_n$ on the $n$ colors (not including $c^*$) acts as a group of automorphisms on $\mathcal{Z}^{c}_n$.  In fact, if $\mathcal{Z}_n$ denotes the ordinary cyclohedron, then very similar considerations as before establish 

\begin{theorem}
\label{cyccolorblind}
$\mathcal{Z}^{c}_n/S_n \cong \mathcal{Z}_n$.
\end{theorem}

Thus $\mathcal{Z}^{c}_n$ covers $\mathcal{Z}_n$. More informally, just as for associahedra, going colorblind on the colorful cyclohedron yields the ordinary cyclohedron.

The automorphism group of the colorful cyclohedron $\mathcal{Z}^{c}_n$ can be derived from that of the colorful associahedron $\mathcal{A}^{c}_n$. In fact, we have the following theorem.

\begin{theorem}\label{colasscentralgroup} 
$\Gamma(\mathcal{Z}^{c}_n) \cong S_{n}\times D_{n+2}$ when $n>1$, and $\Gamma(\mathcal{Z}^{c}_n) \cong D_{6}$ when $n=1$.
\end{theorem}

\begin{proof}
First note that $\mathcal{Z}^{c}_n$ is just a hexagon when $n=1$, and hence $\Gamma(\mathcal{Z}^{c}_1) \cong D_{6}$.

Now assume that $n>1$ and suppose the vertices of $K$ are labeled $1,\ldots,2n+4$, in cyclic order. Then the central diagonals of $K$ are given by $\{i,i+n+2\}$ for $i=1,\ldots,n+2$. 

First note that $\Gamma(K)=D_{2n+4}$ acts in an obvious way on $\mathcal{Z}^{c}_n$. However, since the central inversion $\iota$ of $K$ acts trivially on $\mathcal{Z}^{c}_n$, the subgroup of $\Gamma(\mathcal{Z}^{c}_n)$ induced by $\Gamma(K)$ is isomorphic to $D_{2n+4}/{\langle \iota\rangle} \cong D_{n+2}$ and thus will be denoted by $D_{n+2}$. Hence, since the two subgroups $D_{n+2}$ and $S_n$ of $\Gamma(\mathcal{Z}^{c}_n)$ centralize each other, $\Gamma(\mathcal{Z}^{c}_n)$ must necessarily contain a subgroup isomorphic to $S_{n}\times D_{n+2}$. It remains to show that this is the full group.

Now suppose $\gamma\in\Gamma(\mathcal{Z}^{c}_n)$. Consider the $n+2$ (mutually isomorphic) facets $F^d$ of $\mathcal{Z}^{c}_n$ with central rigid diagonals $d$. These facets are permuted under automorphisms of $\mathcal{Z}^{c}_n$; in fact, by our previous considerations, since $n>1$ they are the only facets of $\mathcal{Z}^{c}_n$ isomorphic to the colorful $n$-associahedron $\mathcal{A}_{n}^c$. In particular this applies to the automorphism $\gamma$.

Consider the central diagonal $e:=\{1,n+3\}$ of $K$, and let $N$ denote the $(n+3)$-gon with vertices $1,\ldots,n+3$ cut off $K$ by $e$. Then $\gamma(F^e)$ is a facet of $\mathcal{Z}^{c}_n$ isomorphic to $F^e$ and hence must be determined by a central diagonal of $K$. Therefore, since $\Gamma(K)$ acts transitively on the central diagonals of $K$, we may assume that, modulo the subgroup $D_{n+2}$ of $\Gamma(\mathcal{Z}^{c}_n)$, this facet is $F^e$ itself; that is, $F^e$ is stabilized by $\gamma$. It follows that $\gamma$ (or rather its restriction to $F^e$) must lie in the automorphism group of $F^e$, which is isomorphic to $\Gamma(\mathcal{A}_{n}^{c})=S_{n}\times \Gamma(N) = S_{n}\times D_{n+3}$. Suppose $\gamma=\alpha\beta$ with $\alpha\in S_n$ and $\beta\in\Gamma(N)$.  (Here, $\alpha$ and $\beta$ are automorphisms of the facet $F^e$, but a priori not of $\mathcal{Z}^{c}_n$ itself.)

Each $(n-1)$-face $G$ of the facet $F^e$ is of the form $G=(C,t)$, with $t$ a vertex of $F^e$ and $C=C(K)\setminus\{c_G\}$ for some $c_{G}\in C(K)$. Let $d_{G}',d_{G}''$ denote the centrally symmetric pair of non-central diagonals of $K$ whose color in $t$ is $c_G$. Flips involved in generating $G$ can only be performed on diagonals distinct from $d_{G}'$, $d_{G}''$ and $e$. However, for reasons similar to those mentioned earlier, an $(n-1)$-face $G$ of $F^e$ is not generally determined by the diagonals $d_{G}',d_{G}''$ and the color $c_G$; it also depends on the distribution of colors over the regions into which $K$ is dissected by $d_{G}'$, $d_{G}''$ and $e$. Note that, up to relabeling, we may always assume that $d_{G}'=d_{i,j}':=\{i,j\}$ and $d_{G}''=d_{i,j}'':=\{i+n+2,j+n+2\}$ for some $i$ and $j$ with $1\leq i < j-1\leq n+2$.

By the diamond condition, each $(n-1)$-face $G$ of $F^e$ lies in exactly one other facet, $F_G$ (say), of $\mathcal{Z}^{c}_n$. Let $\mathcal{F}^e$ denote the family of all facets of $\mathcal{Z}^{c}_n$ distinct from $F^e$ which have an $(n-1)$-face in common with $F^e$. As $\gamma$ and $\beta$ leave $F^e$ invariant, they must necessarily permute the facets in $\mathcal{F}^e$.

For $1\leq i < j-1\leq n+2$, let $\mathcal{F}_{i,j}^e$ denote the family of facets $F_G$ in $\mathcal{F}^e$ with $G$ such that $d_{G}'=d_{i,j}',d_{G}''=d_{i,j}''$. Now observe that the facets in the families $\mathcal{F}_{i,i+2}^e$ ($i=1,\ldots,n+1$) are isomorphic to colorful $n$-cyclohedra $\mathcal{Z}^{c}_{n-1}$, while those in $\mathcal{F}_{2,n+3}^e$ and $\mathcal{F}_{1,n+2}^e$ are not. The combinatorial automorphism of $N$ inducing the automorphism $\beta$ of the facet $F^e$ naturally permutes the $n+3$ diagonals $d_{i,i+2}'$ ($i=1,\ldots,n+1$), $d_{2,n+3}'$ and $d_{1,n+2}'$ of $K$, so $\beta$ must induce a permutation of the $n+3$ families $\mathcal{F}_{i,i+2}^e$ ($i=1,\ldots,n+1$), $\mathcal{F}_{2,n+3}^e$ and $\mathcal{F}_{1,n+2}^e$. But automorphisms of polytopes take facets to isomorphic facets, so $\beta$ must either fix both families $\mathcal{F}_{2,n+3}^e$ and $\mathcal{F}_{1,n+2}^e$, or interchange them. Now, if the families $\mathcal{F}_{2,n+3}^e$ and $\mathcal{F}_{1,n+2}^e$ are interchanged by $\beta$, then the combinatorial automorphism of $N$ underlying $\beta$ is simply the combinatorial reflection $\rho$ on $N$ that takes $e$ to itself.

 Thus $\beta$ is either the identity automorphism of $\mathcal{A}^{c}_n$ or the automorphism of $\mathcal{A}^{c}_n$ induced by $\rho$. 

In the former case we know that $\gamma$ is an automorphism of $\mathcal{Z}^{c}_n$ that acts on the facet $F^e$ like the element $\alpha$ of $S_n$; in particular, $\gamma$ maps colored triangulations in $F^e$ to colored triangulations in $F^e$ with the same support. As we remarked earlier, $\Gamma(\mathcal{Z}^{c}_n)$ contains a subgroup isomorphic to $S_n$ obtained from the natural action of $S_n$ on $\mathcal{Z}^{c}_n$. Now let $\widehat{\alpha}$ denote the element of this subgroup induced by the permutation on $1,\ldots,n$ given by $\alpha$; less formally, $\widehat{\alpha}$ is induced by $\alpha$. Then it is clear that $\widehat{\alpha}^{-1}\gamma$ is an automorphism of $\mathcal{Z}^{c}_n$ that fixes $F^e$ and acts on the entire facet $F^e$ like the identity. Hence, since automorphisms of polytopes are determined by the effect on a flag, $\widehat{\alpha}^{-1}\gamma$ must in fact be the identity on $\mathcal{Z}^{c}_n$. Thus $\gamma=\widehat{\alpha}$ lies in the copy of $S_n$ inside $\Gamma(\mathcal{Z}^{c}_n)$. 

In the latter case we can think of $\rho$ as the restriction to $N$, of the combinatorial reflection $\rho'$ of $K$ that interchanges the vertices of $e$ (and leaves $e$ invariant). It follows that we can view $\beta$ as the automorphism induced on $F^e$ by the combinatorial automorphism $\rho'$ of $K$ (considered modulo $\iota$). Hence, by applying the automorphism of $\mathcal{Z}^{c}_n$ determined by $\rho'$ we can reduce $\gamma$ modulo $D_{n+2}$ to an automorphism of $\mathcal{Z}^{c}_n$ that acts on the facet $F^e$ like the element $\alpha$ of $S_n$. But now we can argue as in the previous case and conclude that $\gamma$ lies in the subgroup $S_n$ of $\Gamma(\mathcal{Z}^{c}_n)$. This completes the proof.
\end{proof}

Note that a simplified version of our proof also shows that the automorphism group of the ordinary $(n+1)$-cyclohedron $\mathcal{Z}_n$ is given by 
\[ \Gamma(\mathcal{Z}_n) \cong D_{n+2}.\]

\noindent
{\em Acknowledgement.}
We are very grateful to the anonymous referees for their careful reading of our original manuscript and for the helpful comments and corrections that have greatly improved our article.


\begin{thebibliography}{99}

\bibitem{colpoly} G.~Araujo-Pardo, I.~Hubard, D.~Oliveros and E.~Schulte, {\em Colorful polytopes and graphs\/}, Israel J. Mathematics {\bf 195} (2013), 647--675.

\bibitem{Bott} R.~Bott and C.~Taubes, {\em On the self-linking of knots\/}, J. Math. Phys. {\bf 35} (1994), 5247--5287. 

\bibitem{brualdi} R.A.~Brualdi, {\em Introductory Combinatorics\/}, 4th edition, Prentice Hall, 2004.

\bibitem{csz} C.~Ceballos, F.~Santos and G.M.~Ziegler, {\em Many non-equivalent realizations of the associahedron\/}, ArXiv 2011.

\bibitem{cfz} F.~Chapoton, S.~Fomin and A.~Zelevinsky, {\em Polytopal realizations of generalized associahedra\/}, Canad. Math. Bull. {\bf 45} (2002), 537--566. 

\bibitem{CL96} G.~Chartrand and L.~Lesniak,  {\em Graphs and Digraphs\/}, 3rd edition, Chapman and Hall, 1996.

\bibitem{fom} S.~Fomin and N.~Reading, {\em Root systems and generalized associahedra\/}, In:\
Geometric Combinatorics (E.~Miller, V.~Reiner and B.~Sturmfels, eds.), IAS/Park City Math. Ser., 13, Amer. Math. Soc., Providence, RI, 2007, pp. 63--131.

\bibitem{fom-zele} S.~Fomin and A.~Zelevinsky, {\em Y-systems and generalized associahedra\/}, Ann. of Math. {\bf 158} (2003), 977--1018. 

\bibitem{lee} C.W.~Lee, {\em The associahedron and triangulations of the $n$-gon\/}, European J.\ Combinatorics {\bf 10} (1989), 551--560.

\bibitem{loday} J.L.~Loday, {\em Realization of the Stasheff polytope\/}, Arch. Math. {\bf 83} (2004), 267--278.

\bibitem{McMS02} P.~McMullen and E.~Schulte,  {\em Abstract Regular Polytopes\/}, Cambridge University Press, 2002.

\bibitem{post} A.~Postnikov, {\em Permutohedra, associahedra, and beyond\/}, Int. Math. Res. Notices, Vol. 2009, No. 6, pp. 1026--1106.

\bibitem{Shinder} S.~Shnider and S.~Sternberg, {\em Quantum Groups: From Coalgebras to Drinfeld Algebras\/}, Graduate Texts in Mathematical Physics, International Press Inc., Boston, 1993.

\bibitem{simion} R.~Simion, {\em A type-B associahedron\/}, Adv. Appl. Math. {\bf 30} (2003), 2--25.

\bibitem{stan}  R.P.~Stanley, {\em Enumerative Combinatorics\/}, Vol. 1, Second Edition, Cambridge Studies in Advanced Mathematics, 49, Cambridge University Press, Cambridge, 2012. 

\bibitem{stash} J.~Stasheff, {\em Homotopy associativity of H-spaces I, II\/}, Trans. Amer. Math. Soc. {\bf 108} (1963), 275--292 and 293--312. 

\bibitem{tamari} D.~Tamari, {\em The algebra of bracketings and their enumeration\/}, Nieuw Archief voor Wiskunde, Ser. 3, {\bf 10}  (1962), 131--146.


\bibitem{Vizing} V.G.~Vizing, {\em On an estimate of the chromatic class of a $p$-graph\/}, Diskret. Analiz. {\bf 3} (1964), 25--30.

\bibitem{Z95} G.~Ziegler, {\em Lectures on Polytopes\/}, Springer-Verlag, New York, 1994.


\end{thebibliography}
\end{document}